\documentclass[10pt,a4paper,english]{scrartcl}
\pdfoutput=1

\usepackage[T1]{fontenc}
\usepackage[utf8]{inputenc}
\usepackage{babel}

\usepackage[sc]{mathpazo}   %sc for real small caps
\usepackage[scaled=.95]{helvet}
\usepackage{courier}

\usepackage{ltablex}

\usepackage{microtype}

\KOMAoptions{DIV=last}						% Recalculate type area

\usepackage{amssymb,amsthm, amsmath,mathtools}		% math packages

\usepackage{dsfont}						% naturals, integers,...

\usepackage{tikz}							% drawings
\usetikzlibrary[matrix,arrows,calc,intersections, decorations.pathreplacing]
\usepackage{enumitem}

\usepackage{hyperref}
\hypersetup{
  colorlinks   = true,          
  urlcolor     = blue,    
  linkcolor    = blue,       
  citecolor   = purple        
}

\usepackage{todonotes}

\newcommand{\xqedhere}[2]{%
  \rlap{\hbox to#1{\hfil\llap{\ensuremath{#2}}}}}

\let\OLDthebibliography\thebibliography
\renewcommand\thebibliography[1]{
  \OLDthebibliography{#1}
  \small
  \setlength{\parskip}{0pt}
  \setlength{\itemsep}{1.5pt plus 0.3ex}
}

 \theoremstyle{plain}
 \newtheorem{thm}{Theorem}[section]
  \theoremstyle{definition}
  \newtheorem{defn}[thm]{Definition}
  \theoremstyle{definition}
  \newtheorem{example}[thm]{Example}
  \theoremstyle{plain}
  
  \theoremstyle{plain}
  \newtheorem{cor}[thm]{Corollary}
  \theoremstyle{remark}
  \newtheorem{rem}[thm]{Remark}
  \theoremstyle{plain}
  \newtheorem{prop}[thm]{Proposition}
  \theoremstyle{plain}
  
  \theoremstyle{remark}
  
  \theoremstyle{remark}
  
  \theoremstyle{plain}
  
  \theoremstyle{definition}
  
  \numberwithin{equation}{section}

\DeclareMathOperator{\Td}{Td}

\DeclareMathOperator{\Spec}{Spec}
\DeclareMathOperator{\Trop}{Trop}
\DeclareMathOperator{\trop}{trop}

\DeclareMathOperator{\Newt}{Newt}

\DeclareMathOperator{\rk}{rk}

\DeclareMathOperator{\conv}{conv}

\DeclareMathOperator{\id}{id}

\DeclareMathOperator{\Mink}{M}

\DeclareMathOperator{\Lat}{Lat}
\DeclareMathOperator{\ch}{ch}
\DeclareMathOperator{\vol}{vol}

\newcommand{\C}{{\mathds C}}
\newcommand{\G}{{\mathds G}}

\newcommand{\Q}{{\mathds Q}}
\newcommand{\Z}{{\mathds Z}}

\newcommand{\R}{{\mathds R}}

\newcommand{\mF}{{\mathcal F}}
\newcommand{\mG}{{\mathcal G}}
\newcommand{\mL}{{\mathcal L}}
\newcommand{\mN}{{\mathcal N}}
\newcommand{\mO}{{\mathcal O}}
\newcommand{\mT}{{\mathcal T}}

\newcommand{\bI}{{\mathbf I}}

\renewcommand{\injlim}{\varinjlim}
\renewcommand{\top}{{\operatorname{top}}}

\addtokomafont{section}{\large\rmfamily}
\addtokomafont{subsection}{\normalsize\rmfamily}
\addtokomafont{title}{\rmfamily\mdseries\scshape}	
\addtokomafont{paragraph}{\rmfamily}

\begin{document}

\title{\large Refined Tropicalizations for Schön Subvarieties of Tori}

\author{\normalsize Andreas Gross}
\date{}
\maketitle

\begin{abstract}
We introduce a relative refined $\chi_y$-genus for schön subvarieties of algebraic tori. These are rational functions of degree minus the codimension with coefficients in the ring of lattice polytopes. We prove that the relative refined $\chi_y$ turns sufficiently generic intersections into products, and that we can recover the ordinary $\chi_y$-genus by counting lattice points. Applying the tropical Chern character to the relative refined $\chi_y$-genus we obtain a refined tropicalization which is a tropical cycle having rational functions with $\Q$-coefficients as weights. We prove that the top-dimensional component of the refined tropicalization specializes to the unrefined tropicalization up to sign when setting $y=0$ and show that we can recover the $\chi_y$-genus by integrating the refined tropicalization with respect to a Todd measure.
\end{abstract}

\section{Introduction}
One of the main concerns of tropical geometry is to determine which invariants of a subvariety of an algebraic torus are preserved when passing to its tropicalization. Motivated by the conjectured correspondence theorem for refined curve counts \cite{BG16}, we ask if the $\chi_y$-genus is among the preserved invariants, or rather if we can define a refinement of the tropicalization from which it can easily be recovered. In \cite{NPS16}, good progress has been made for the refined correspondence theorem in genus $1$ by writing the universal curve over a linear system on a toric surface as a generic complete intersection. Tropicalizations of generic intersections can be computed using tropical intersection theory \cite{OP13}, a statement that we will show to remain true in the refined setting. This gives a first concrete idea of the role tropical intersection theory might play in the theory of refined curve counting.
%, a role that has so far only been hinted at by the independence of the refined tropical curve counts of chosen point configurations \cite{IM13}.

The question whether the $\chi_y$-genus of a sufficiently generic complete intersection $Z=V(f_1)\cap \ldots \cap V(f_k)$ in $(\C^*)^n$ can be computed from combinatorial data associated to the $f_i$ is not new. It has already been known since the work of Danilov and Khovanski\u\i\ \cite{DK86} (cf.\ \cite{KS16} for a tropical proof) how to compute the $\chi_y$-genus of $Z$ in terms of the Newton polytopes $\Delta_i$ of the Laurent polynomials $f_i$. In the case $k=1$ Danilov and Khovanski\u\i\ gave an explicit formula for $\chi_y(Z)$ in terms of the $h^*$-polynomial of $\Delta_1$ and suggested to use the Cayley trick in order to reduce the computation of $\chi_y$ for complete intersections to the hypersurface case. While this certainly allows to compute the $\chi_y$-genus of any given complete intersection in $(\C^*)^n$, deducing a closed formula for $\chi_y(Z)$ in terms of the $\Delta_i$ from this algorithm still requires serious effort. This has been worked out much more recently by Di Rocco, Haase, and Nill \cite{DHN16} who proved that the $p$-th coefficient of $\chi_y(Z)$ is equal to 

\begin{equation*}
(-1)^{n-p}
\sum_{I\subseteq \{1,\dots, k\}}
(-1)^{|I|}
\left(
\sum_{\substack{\beta\in\Z^I_{\geq 0}\\
						|\beta|\leq p}}
(-1)^{|\beta|}
\binom{n+|I|}{p-|\beta|}
\left|
\left(
\sum_{i\in I} 
(1+\beta_i)\Delta_i
\right)
\cap \Z^n
\right|				
\right) \ ,
\end{equation*}
where the last occurring summation symbol denotes a Minkowski sum of polytopes. 

%The main motivation for this paper comes from the question whether the formula above can be expressed using tropical intersection theory, that is whether there is a refined tropicalization from which one can recover the $\chi_y$-genus and which turns generic intersections into tropical intersection products. 
In order to see how this formula can be expressed using tropical intersection theory, we need to    simplify it considerably. The major problem in any attempt of simplification is the occurrence of counts of lattice points of various Minkowski sums of the polytopes $\Delta_i$. The easiest way to deal with this problem is to extend scalars: instead of working over the integers, we can work over the commutative ring $\mL(\Z^n)$ of lattice polytopes. It is generated by elements $[P]$ associated to lattice polytopes $P\subseteq \R^n$ which satisfy $[P]\cdot[Q] =[P+Q]$, that is multiplication is given by Minkowski sums. Furthermore, there exists a morphism $\Lat\colon \mL(\Z^n)\to \Z$ of abelian groups that sends $[P]$ to $|P\cap \Z^n|$ for every lattice polytope $P$. We see that after the substitution 
\begin{equation*}
\left|\left(\sum_{i\in I} (1+\beta_i)\Delta_i\right)\cap \Z^n \right| \qquad\longleftrightarrow\qquad \prod_{i\in I} [\Delta_i]^{1+\beta_i}
\end{equation*}
in the formula above, the polynomial $\chi_y(Z)\in \Z[y]$ is the image of the power series
\begin{equation*}
\sum_{p=0}^\infty
(-1)^{n-p}
\sum_{I\subseteq \{1,\dots, k\}}
(-1)^{|I|}
\left(
\sum_{\substack{\beta\in\Z^I_{\geq 0}\\
						|\beta|\leq p}}
(-1)^{|\beta|}
\binom{n+|I|}{p-|\beta|}
\prod_{i\in I} 
[\Delta_i]^{1+\beta_i}
\right)y^p
\end{equation*}
in $\mL(\Z^n)[[y]]$ under coefficient-wise application of $\Lat$. This series factors beautifully as
\begin{equation*}
(y-1)^n\prod_{i=1}^k \frac{1-[\Delta_i]}{1-[\Delta_i]y}\ ,
\end{equation*}
an expression astonishing not only for its simplicity, but also for almost being intersection-theoretic, making a connection to tropical intersection theory much more likely:  to compute the $\chi_y$-genus of an intersection $\bigcap V(f_i)$ we start with the $\chi_y$-genus of the ambient torus, i.e.\ with $(y-1)^n$, and then add one factor for every hypersurface we cut with. Of course, this is true only before applying $\Lat$, which does not preserve products. We are thus tempted us to consider $\frac{1-[\Delta_i]}{1-[\Delta_i]y}$ as refinement of $\chi_y(V(f_i))$ relative to the embedding of $V(f_i)$ in $T$. 

In the course of this paper we will explain and generalize the identity
\begin{equation*}
\chi_y\left(\bigcap V(f_i)\right)=\Lat \left((y-1)^n\prod_{i=1}^k \frac{1-[\Delta_i]}{1-[\Delta_i]y}\right) \ .
\end{equation*}

For the generalization we will need to replace the condition ``generic with Newton polytope $\Delta$''  for hypersurfaces by a condition that works in any codimension. A suitable notion is what was called \emph{schön} by Tevelev \cite{Tev07}, a condition that has already been used in \cite{KS12,KS16} to study the motivic nearby fiber of subvarieties of tori. The following theorem summarizes the results of section \ref{sec:results}.

\begin{thm}
\label{thm:main theorem}
We can assign to every schön subvariety $Z\subseteq T=(\C^*)^n$ a rational function 
\begin{equation*}
\chi_y^{T}(Z)\in \frac 1{(y-1)^n}\mL(\Z^n)[y]\subseteq \mL(\Z^n)[[y]]
\end{equation*}
such that 
\begin{enumerate}[label=(\roman*)]
\item we have 
\begin{equation*}
\chi_y(Z)=\chi_y(T)\Lat\left( \chi_y^T(Z)\right)=(y-1)^n\Lat\left(\chi_y^T(Z)\right)
\end{equation*}
for every schön subvariety $Z$ of $T$.
\item we have 
\begin{equation*}
\chi_y^T(V(f))=\frac{1-[\Delta]}{1-[\Delta]y}
\end{equation*}
for every Laurent polynomial $f\in \C[\Z^n]$ which is generic among polynomials with Newton polytope $\Delta$. 
\item for two schön subvarieties $Z$ and $Z'$ intersecting generically we have 
\begin{equation*}
\chi^T_y(Z\cap Z')=\chi^T_y(Z)\cdot \chi^T_y(Z') \ .
\end{equation*}
Here, intersecting generically means that $Z\cap Z'$ is schön, and the equality  $\overline {Z\cap Z'} =\overline Z\cap \overline Z'$ holds in any toric variety with sufficiently fine fan.
\end{enumerate}
\end{thm}

To prove the theorem we use Morelli's description of the Grothendieck rings of vector bundles on smooth toric varieties as subrings of the ring of lattice polytopes \cite{Mor93}. This description makes it natural to apply the tropical version of the Chern character to $\chi_y^T(Z)$ and thereby obtain a tropical cycle having rational functions in $y$ with $\Q$-coefficients as weights. We will call this cycle the \emph{refined tropicalization} of $Z$ and denote it by $\Trop_y(Z)$. Since the Chern character is a ring homomorphism, part (iii) of Theorem \ref{thm:main theorem} immediately translates to the statement that the refined tropicalization of a generic intersection equals the intersection (product) of the refined tropicalizations, that is
\begin{equation*}
\Trop_y(Z\cap Z')= \Trop_y(Z)\cdot \Trop_y(Z')
\end{equation*}
for $Z$ and $Z'$ intersecting generically. This behaviour is analagous to the unrefined case \cite{OP13}, and in fact we prove that the top-dimensional part of $\Trop_y(Z)$ specializes to the unrefined tropicalization $\Trop(Z)$ up to sign by setting $y=0$. More precisely, we prove in Proposition \ref{prop:Refined Trop specializes to unrefined Trop} that there is an equality 
\begin{equation*}
|\Trop_y(Z)|=|\Trop(Z)|
\end{equation*} 
of supports and we have
\begin{equation*}
\Trop_0(Z)^\top=(-1)^k\Trop(Z) \ , 
\end{equation*}
where $k$ is the codimension of $Z$ and $\Trop_y(Z)^\top$ denotes the top-dimensional component of $\Trop_y(Z)$. We conclude the paper by explaining how Hirzebruch-Riemann-Roch allows one to recover $\chi_y(Z)$ from $\Trop_y(Z)$ by integrating it with respect to a Todd measure $\mu$. This is the translation of the first part of Theorem \ref{thm:main theorem} which, in the notation of Corollary \ref{cor: integral over refined trop yields chiy} reads
\begin{equation*}
\chi_y(Z)=\chi_y(T)\cdot \int \Trop_y(Z) \,d\mu \ .
\end{equation*}

\paragraph*{Acknowledgment} The author would like to thank Johannes Nicaise and Sam Payne for helpful conversations. This research was  supported by the ERC Starting Grant MOTZETA (project 306610) of the European Research Council (PI: Johannes Nicaise).

\section{Preliminaries}
We quickly recall the main results and definitions needed in this paper. Throughout this paper, we will work over an algebraically closed field $\kappa$ of characteristic $0$.

\subsection*{The $\chi_y$-genus}
If $X$ is a smooth projective variety, its $\chi_y$-genus can be defined as 
\begin{equation*}
\chi_y(X)=\sum_{p=0}^{\dim X} (-1)^p \chi(X,\Omega^p_X)y^p \ .
\end{equation*}
It can be shown that there is a unique way to extend this definition to all algebraic schemes such that $\chi_y$ is additive with respect to stratifications, that is there is a unique way to assign to every algebraic scheme $X$ a polynomial $\chi_y(X)\in \Z[y]$ such that for smooth projective varieties the above formula holds and for every closed immersion $X\to Y$ we have
\begin{equation*}
\chi_y(X)=\chi_y(Y)+\chi_y(X\setminus Y) \ .
\end{equation*}
The $\chi_y$-genera are also multiplicative, that is $\chi_y(X\times Y)=\chi_y(X)\cdot \chi_y(Y)$ for all algebraic schemes $X$ and $Y$ (the product of course taken over the base field), although we will not need this here.

\subsection*{The Grothendieck groups $K^0$ and $K_0$}
For any scheme $X$ 	one can define the Grothendieck group of vector bundles $K^0(X)$. It is the free abelian group generated by all isomorphism classes of vector bundles, modulo the relations $[\mF]=[\mF']+[\mF'']$, whenever $\mF$ is an extension of $\mF'$ by $\mF''$ (here, $[\mF]$ denotes the generator associated to the isomorphism class of $\mF$). There is a natural ring structure on $K^0(X)$ defined by the product law $[\mF]\cdot[\mG]=[\mF\otimes_{\mO_X} \mG]$, making $K^0(X)$ a commutative ring with unit $[\mO_X]$. In fact, $K^0(X)$ even is a $\lambda$-ring: for every vector bundle $\mF$ we can define a polynomial $\lambda_t(\mF)=\sum_k [\bigwedge\nolimits^k\mF]y^k\in K^0(X)[y]$. This extends uniquely to a morphism of abelian groups $K^0(X)\to 1+ yK^0(X)[[y]]$ satisfying various other good properties which we will not need.

Similarly, one can define the Grothendieck group of coherent sheaves $K_0(X)$, which is the free abelian group generated by all isomorphism classes of coherent sheaves, again modulo the relations coming from extension. Tensoring with an arbitrary coherent sheaf is not exact, so $K_0(X)$ is not a ring, but it is easily seen to be a $K^0(X)$-module. 

The Grothendieck groups are universal among pairs of abelian groups and maps from the set of isomorphism classes of finite rank locally free / coherent sheaves which are additive on exact sequences into them. In particular, as taking Euler characteristics of coherent sheaves on projective varieties is additive, for projective $X$ there is a unique function $\chi\colon K_0(X)\to \Z$ such that $\chi[\mF]=\chi(X,\mF)$ for every coherent sheaf $\mF$.

For any morphism $f\colon X\to Y$ of schemes, there is a map pull-back $f^*\colon K^0(Y)\to K^0(X)$ mapping $[\mF]$ to $[f^*\mF]$. This is a morphisms of rings. If $f$ is proper, there also is a push-forward $f_*\colon K_0(X)\to K_0(Y)$ mapping $[\mF]$ to $\sum_i (-1)^i[R^i f_*\mF]$. This is a morphism of $K^0(X)$-modules, or equivalently, there is a projection formula.  We will only need the push-forward for closed embeddings, in which case the push-forward has the simple form $f_*[\mF]=[f_*\mF]$.

There is a canonical morphism
\begin{equation*}
K^0(X)\to K_0(X)
\end{equation*}
of $K^0(X)$-modules, sending $[\mF]$ to $[\mF]$ for every finite rank locally free sheaf $\mF$. For smooth varieties, this is always an isomorphism.
Since we only consider smooth varieties in this note, we will not distinguish between $K^0(X)$ and $K_0(X)$ and simply write $K(X)$

\subsection*{The Grothendieck ring of smooth complete toric varieties}
\label{subsec:KX for toric variety}
The Grothendieck ring $K(X)$ of a smooth complete toric variety $X$ can be described with polyhedral geometry. Let $T$ be an algebraic torus and with character lattice $M$, and let $X$ be a toric variety with big open torus $T$. In \cite{Mor93}, Morelli constructs an injective ring homomorphism  $\bI_X\colon K(X)\to\mL(M)$ into the ring of lattice polytopes. The underlying abelian group of $\mL(M)$ is the free abelian group on all lattice polytopes in $M_\R=M\otimes_\Z \R$, modulo the scissors relations $[P]+[Q]=[P\cup Q]+[P\cap Q]$, whenever $P$, $Q$, $P\cup Q$, and $P\cap Q$ are all lattice polytopes, and the translation relations $[P+m]=[P]$ for all lattice polytopes $P$ and $m\in M$. Here, $[P]$ denotes the generator of $\mL(M)$ corresponding to the lattice polytope $P$. The multiplication on $\mL(M)$ is given by the Minkowski sum, that is $[P]\cdot [Q]=[P+Q]$. We will not need a precise description of $\bI_X$, besides the following three facts:
\begin{itemize}
\item If $D$ is an ample torus-invariant divisor on $X$ with associated polytope $P_D$, then $\bI_X[\mO_X(D)]=[P_D]$.

\item For every coherent sheaf $\mF$ on $X$ we have $\chi(X,\mF)= \Lat \circ \bI_X[\mF]$, where $\Lat\colon \mL(M)\to \Z$ is the group homomorphism sending $[P]$ to its number $|P\cap M|$ of lattice points.

\item The morphisms $\bI_X$ are functorial with respect to toric morphisms. In particular, whenever $f\colon X'\to X$ is a toric modification from another smooth complete toric variety $X'$ with big open torus $T$, we have $\bI_{X'}\circ f^*(\alpha)=\bI_X(\alpha)$ for all $\alpha\in K(X)$.
\end{itemize}

\subsection*{Schön subvarieties of tori}

Let $Z$ be a subvariety of an algbraic torus $T$. Then for every toric variety $X$ with big open torus $T$ we can consider the closure $\overline Z$ in $X$. The pair $(X,Z)$ is called tropical, if $\overline Z$ is proper and the multiplication map $T\times \overline Z\to X$ is faithfully flat. Tevelev showed in \cite{Tev07} that for given $Z$, there always exists a toric variety $X$ such that $(X,Z)$ is tropical. Furthermore, he showed that the set of such $X$ is stable under toric modification. However, except in trivial cases $X$ can never be proper itself. So let us call a pair $(X,Z)$ \emph{almost tropical} if $\overline Z$ is proper and the multiplication map is  flat, but not necessarily surjective. Then again, the set of all $X$ such that $(X,Z)$ is almost tropical is closed under toric modification of $X$, and therefore we can always choose $X$ to be smooth and projective.

A subvariety $Z$ of $T$ is called schön, if there \emph{exists} a toric variety $X$ such that $(X,Z)$ is tropical, and the multiplication map is smooth. By \cite[Thm.\ 1.4]{Tev07}, this implies that the multiplication map is smooth for \emph{all} choices of $X$ for which $(X,Z)$ is almost tropical. If $Z$ is schön and $(X,Z)$ is almost tropical, then we also say that $\overline Z$ is a schön subvariety of $X$.

\section[The relative refined chi-y-genus]{The relative refined $\chi_y$-genus}
\label{sec:results}

As we have already seen, the $\chi_y$-genus of a smooth projective variety can be computed in terms of Euler characteristics of its sheaves of forms. Since we are mainly interested in non-projective varieties, it will be convenient to have a formula for the $\chi_y$-genus of a non-proper variety involving Euler characteristics of bundles on a suitable compactification.

\begin{prop}
\label{prop:formula for chiy in terms of log cotangent complex}
Let $D$ be a simple normal crossings divisor on a smooth projective variety $X$ of pure dimension $n$. Then we have the equality
\begin{equation*}
\chi_y(X\setminus D) = \sum_{k=0}^n (-1)^{n-k}\chi\big(X, \Omega_X^{n-k}(\log D) \big)y^k \ .
\end{equation*}
\end{prop}
\begin{proof}
We proof this by induction, first on $n$, then on the number $l$ of irreducible components of $D$. For $n=0$ the statement is trivial, so assume $n>0$. If $l=0$, we have $\Omega_X^1=\Omega_X^1(\log D)$ and the formula follows from Serre duality. So assume $l>0$. Let $D_1$ be an irreducible component of $D$. For every $p>0$ we have a short exact sequence (see \cite[Property 2.3 b)]{EV92})
\begin{equation*}
0\to \Omega^p_X(\log(D-D_1)) \to \Omega^p_X(\log D)\to \Omega^{p-1}_{D_1}(\log(D-D_1)|_{D_1})\to 0 \ .
\end{equation*}
This yields 
\begin{multline*}
\sum_{k=0}^n (-1)^{n-k}\chi\big(X, \Omega_X^{n-k}(\log D) \big)y^k= \\
=\sum_{k=0}^n (-1)^{n-k}\chi\big(X,\Omega^{n-k}_X(\log (D-D_1))\big) \\
- \sum_{k=0}^{n-1} (-1)^{(n-1)-k}\chi\big(D_1,\Omega^{(n-1)-k}_{D_1}(\log(D-D_1)|_{D_1})\big)
\end{multline*}
which is, by induction, equal to
\begin{equation*}
\chi_y(X\setminus (D-D_1))- \chi_y(D_1\setminus (D-D_1)) =\chi_y(X\setminus D) \ ,
\end{equation*}
where the last equality follows from the additivity of $\chi_y$.
\end{proof}

Motivated by the formula for the $\chi_y$-genus in the preceding proposition and the conormal sequence for the log-cotangent bundles, we make the following definition. It is not very enlightening by itself, but we will quickly see that it has the properties we desire.

\begin{defn}
Let $Z$ be a codimension-$k$ schön subvariety of a smooth projective toric variety $X$. Let $\mN_{Z/X}^\vee$ be the conormal bundle on $Z$, and let $i\colon Z\to X$ be the inclusion. We define
\begin{multline*}
\chi_y^X(Z) \coloneqq \bI_X\circ i_*\bigg[\Big(y^k\lambda_{-y^{-1}}[\mN_{Z/X}^\vee]\Big)^{-1}\bigg]=\\
= \bI_X\circ i_*\left[ \left(\sum_{i=0}^{n} (-1)^{k-i}\big[\bigwedge\nolimits^{k-i}\mN_{Z/X}^\vee\big]\right)^{-1}\right] \in \mL(M)[[y]] \ .
\end{multline*}
\end{defn}

The following proposition shows that the relative refined $\chi_y$-genus deserves its name. 

\begin{prop}
\label{prop:chiy deserves its name}
Let $Z$ be a codimension-$k$ schön subvariety of a smooth projective toric variety $X$ with big open torus $T\cong \G_m^n$. Then $(y-1)^n\chi_y^X(Z)$ is a polynomial of degree $n-k$, and  we have
\begin{equation*}
 \chi_y(Z\cap T) = \chi_y(T) \Lat\left(\chi_y^X(Z)\right)=(y-1)^n \Lat\left(\chi_y^X(Z)\right) \ .
\end{equation*}
\end{prop}
\begin{proof}
Let $D$ denote the boundary $X\setminus T$. It is a simple normal crossings divisor, and since $Z$ is schön, the same is true for $D|_Z$. For the conormal sequence for the log cotangent bundles 
\begin{equation*}
0\to \mN_{Z/X}^\vee \to i^*\Omega^1_X(\log D) \to \Omega^1_Z(\log D|_Z) \to 0 \ ,
\end{equation*}
where $i\colon Z\to X$ is the inclusion map, we deduce that 
\begin{equation*}
\left(y^k\lambda_{-y^{-1}} [\mN_{Z/X}^\vee]\right)^{-1}\cdot y^{n-k} i^*\left(\lambda_{-y^{-1}}[\Omega^1_X(\log D)]\right)  =  y^k \lambda_{-y^{-1}}[\Omega^1_Z(\log D|_Z)]	\ .
\end{equation*}
Taking $i_*$ of both sides and using the projection formula yields
\begin{equation*}
i_* \left[\left(y^{k}\lambda_{-y^{-1}} [\mN_{Z/X}^\vee]\right)^{-1}\right]\cdot y^n \lambda_{-y^{-1}}[\Omega^1_X(\log D)]  
= i_*\left[ y^{n-k} \lambda_{-y^{-1}}[\Omega^1_Z(\log D|_Z)\right] \ .
\end{equation*}
But $\Omega^1_X(\log D)$ is trivial of rank $n$ \cite[Example 8.1.2]{CLS}, and hence
\begin{equation*}
y^n\lambda_{-y^{-1}}[\Omega^1_X(\log D)]=y^n (\lambda_{-y^{-1}}[\mO_X])^n=y^n(1-y^{-1})^n=(y-1)^n \ .
\end{equation*}
It follows that 
\begin{equation}
\label{equation 1}
(y-1)^n \cdot i_* \left[\left(y^{k}\lambda_{-y^{-1}} [\mN_{Z/X}^\vee]\right)^{-1}\right]=
i_*\left[ y^{n-k} \lambda_{-y^{-1}}[\Omega^1_Z(\log D|_Z)\right] \ ,
\end{equation}
which clearly is a polynomial of degree $n-k$ in $K_0(X)[y]$. As $\bI_X$ is a ring homomorphism, the first claim follows. For the second claim we recall that for any vector bundle $\mF$ on $Z$ we have $\Lat(\bI_X(i_*[\mF]))=\chi(Z,\mF)$. Thus, applying $\Lat\circ \bI_X$ to both sides of equation \ref{equation 1} yields
\begin{equation*}
\Lat\left((y-1)^n\chi_y^X(Z)\right)=\sum_{i=0}^{n-k} (-1)^{n-k-i} \chi\big(Z,\Omega^{n-k-i}_Z(\log D|_Z)\big) \ ,
\end{equation*}
which is equal to $\chi_y(Z\cap T)$ by Proposition \ref{prop:formula for chiy in terms of log cotangent complex}. Because $\Lat$ is $\Z$-linear, we may pull $(y-1)^n$ to the front, finishing the proof.
\end{proof}

Until now, the relative refined $\chi_y$-genus is defined for schön subvarieties of smooth projective toric varieties and not for schön subvarieties of tori. The following proposition shows that to define it for a schön subvariety of an algebraic torus, we can simply take its closure in a suitable toric variety and then take the relative refined $\chi_y$-genus of the compactification.

\begin{prop}
\label{prop:independence of chiy}
Let $Z$ be a schön subvariety of a torus $T$, and let $X$ and $X'$ be two smooth projective toric varieties such that both $(X,Z)$ and $(X',Z)$ are almost tropical. Let $\overline Z$ and $\overline Z'$ denote the closures of $Z$ in $X$ and $X'$, respectively. Then we have
\begin{equation*}
\chi_y^X(\overline Z)=\chi_y^{X'}(\overline Z') \ .
\end{equation*}
\end{prop}

\begin{proof}
The set of fans $\Sigma$ for which $X_\Sigma$ is smooth and projective and $(X_\Sigma,Z)$ is almost tropical is directed, so we may assume that $X'$ is a toric modification of $X$, i.e.\ there is a toric morphism $f\colon X'\to X$ induced by a refinement of fans. Denoting the induced morphism $\overline Z'\to \overline Z$ by $g$, and the inclusions of $\overline Z$ and $\overline Z'$ into $X$ and $X'$ by $i$ and $i'$, respectively, we have a diagram
\begin{center}
\begin{tikzpicture}
\matrix[matrix of math nodes, row sep= 5ex, column sep= 4em, text height=1.5ex, text depth= .25ex]{
%%% row 1
|(Z')|  \overline Z'		& 
|(Z'xT)|  \overline Z'\times T 	&
|(X')| X'		\\
%%% row 2	
 |(Z)| \overline Z 	&
|(ZxT)| \overline Z\times T		& 
|(X)|	X\\
};
\begin{scope}[->,font=\footnotesize,auto]
%%% horizontal
%% row 1
\draw (Z') --node{$\id\times 1$} (Z'xT);
\draw (Z'xT) -- (X');
\draw (Z') to[out=35, in=145]node{$i'$} (X');
%% row 2
\draw (Z) -- node{$\id\times 1$} (ZxT);
\draw (ZxT) -- (X);
\draw (Z) to[out=-35, in=-145]node[swap]{$i$} (X);
%%% vertical
\draw (Z')--node{$g$} (Z);
\draw (Z'xT)--node{$g\times \id$} (ZxT);
\draw (X')-- node{$f$}(X);
\end{scope}
\end{tikzpicture}
\end{center}
in which the two middle squares are cartesian by \cite[Prop.\ 2.5]{Tev07}. Since the multiplication maps are smooth, $\overline Z\times T$ and $X'$ are Tor-independent over $X$. Clearly, $\overline Z$ and $\overline Z'\times T$ are also Tor-independent over $\overline Z\times T$. Together, this yields that $\overline Z$ and $X'$ are Tor-independent over $X$ \cite[III 1.5.1]{SGA6}. By \cite[IV 3.1.1]{SGA6}, this implies that
\begin{equation*}
f^*i_*\alpha=i'_*g^*\alpha
\end{equation*} for all $\alpha\in K^0(\overline Z)$. Using that $g^*\mN_{\overline Z/X}^\vee=\mN_{\overline Z'/X'}^\vee$ and $\bI_{X'}\circ f^* = \bI_X$ finishes the proof.
%\begin{align*}
%\chi_y^{X'}(\overline Z')&=P_{X'}\left(i'_*\bigg(\Big(y^k\lambda_{-t^{-1}}[\mathcal N'^\vee]\Big)^{-1}\bigg)\right)\\
%&=P_{X'}\left(i'_*g^*\bigg(\Big(t^k\lambda_{-t^{-1}}[\mathcal N^\vee]\Big)^{-1}\bigg)\right)\\
%&=P_{X'}\left(f^*i_*\bigg(\Big(t^k\lambda_{-t^{-1}}[\mathcal N^\vee]\Big)^{-1}\bigg)\right)\\
%&=P_X\left(i_*\bigg(\Big(t^k\lambda_{-t^{-1}}[\mathcal N^\vee]\Big)^{-1}\bigg)\right)\\
%&=\chi_y^X(\overline Z)
%\end{align*}
\end{proof}

\begin{defn}
Let $Z$ be a schön subvariety of $T$. Then we define
\begin{equation*}
\chi^T_y(Z)\coloneqq\chi^X_y(\overline Z) \ ,
\end{equation*}
where $\overline Z$ is the closure of $Z$ in a smooth projective toric variety $X$ such that $(X,Z)$ is almost tropical. By Proposition \ref{prop:independence of chiy}, this is well-defined.
\end{defn}

Part (i) of Theorem \ref{thm:main theorem} now follows directly from Proposition \ref{prop:chiy deserves its name}.

\begin{cor}
\label{cor:chiy-formula for schoen subvarieties}
Let $Z$ be a schön subvariety of an algebraic torus $T$. Then $\chi_y(T)\ddot \chi_y^T(Z)$ is a polynomial of degree $\dim(Z)$, and we have
\begin{equation*}
\chi_y(Z)=\chi_y(T)\Lat\left(\chi_y^T(Z)\right) \ .
\end{equation*}
\end{cor}

The following proposition proves part (ii) of Theorem \ref{thm:main theorem}.

\begin{prop}
Let $M$ be a lattice, let $\Delta\subset M_\R$ be a lattice polytope, and let $f\in \kappa[M]$ be generic with Newton polytope $\Delta$, then 
\begin{equation*}
\chi_y^T(V(f))= \frac{1-[\Delta]}{1-[\Delta]y} \ ,
\end{equation*}
where $V(f)$ denotes the closed subscheme of $T=\Spec \kappa[M]$ defined by $f$.
\end{prop}

\begin{proof}
Let $\Sigma$ be a smooth projective refinement of the inner normal fan of $\Delta$, and let $X=X_\Sigma$ denote the associated projective toric variety. Then the pair $(X_\Sigma, \overline {V(f)})$ is almost tropical by \cite[Thm.\ 1.5]{LQ11}. Denote $D=\overline{V(f)}$.  The conormal bundle $\mN_{D/X}^\vee$ is equal to $\mO_D(-D)$. Thus, we have
\begin{multline*}
\left( y\lambda_{-y^{-1}}[\mN_{D/X}^\vee]\right)^{-1}= \left(-[\mO_D(-D)]+[\mO_D]y\right)^{-1}=\\
= -[\mO_D(D)]\cdot\sum_{k=0}^\infty [\mO_D(kD)]y^k=-\sum_{k=0}^\infty[\mO_D((k+1)D)] \ .
\end{multline*}
Using the exact sequence
\begin{equation*}
0\to \mO_X(kD)\to \mO_X((k+1)D)\to \mO_D((k+1)D)\to 0
\end{equation*}
we see that the push-forward of $\left( y\lambda_{-y^{-1}}[\mN_{D/X}^\vee]\right)^{-1}$ to $X$ is equal to
\begin{equation*}
\sum_{k=0}^\infty \Big([\mO_X(kD)]-[\mO_X((k+1)D)]\Big)y^k =\frac{1-[\mO_X(D)]}{1-[\mO_X(D)]y} \ .
\end{equation*}
To finish the proof, we note that $\mO_X(D)$ is isomorphic to the toric vector bundle $\mO_X(\Delta)$, which is mapped to $[\Delta]$ by $\bI_X$.
\end{proof}

\begin{rem}
\label{rem:relative chiy of hypersurface as laurent polynomial in y-1}
Without the use of Corollary \ref{cor:chiy-formula for schoen subvarieties}, it is not immediate that $\frac{1-[\Delta]}{1-[\Delta]y}$ is contained in $\mL(M)[y,(y-1)^{-1}]$. To see this directly, we note that by \cite[Lemma 13]{McMullen89} and \cite[Thm.\ 9.1]{McMullen09} we have $(1-[\Delta])^{n+1}=0$ in $\mL(M)$, where $n=\rk M$. Therefore, $[\Delta]$ is a unit and we have
\begin{equation*}
\frac{1-[\Delta]}{1-[\Delta]y}=\frac{1-[\Delta]}{[\Delta](1-y)} \cdot \frac 1{\left(1-\frac{1-[\Delta]}{[\Delta](y-1)}\right)^{-1}}=-\sum_{i=1}^n  \left(\frac {[\Delta]^{-1}-1}{y-1}\right)^i \ .
\end{equation*}
\end{rem}

Finally, we prove part (iii) of Theorem \ref{thm:main theorem}. The formula for the $\chi_y$-genera of complete intersections in algebraic tori will be an immediate consequence.

\begin{prop}
Let $Y$ and $Z$ be two schön subvarieties of a torus $T$ intersecting generically, by which we mean that $Z\cap Z'$ is schön, and there is an equality  $\overline {Z\cap Z'} =\overline Z\cap \overline Z'$ in any toric variety with sufficiently fine fan. Note that the latter condition is equivalent to the existence of a toric variety $X$ such that all three pairs $(X,\overline Y)$, $(X,\overline Z)$, and $(X,\overline Y\cap \overline Z)$ are all almost tropical, and the equality $\overline {Z\cap Z'} =\overline Z\cap \overline Z'$ holds on $X$. Then we have
\begin{equation*}
\chi_y^T(Y\cap Z) =\chi_y^T(Y)\cdot\chi_y^T(Z) \ .
\end{equation*}
\end{prop}
\begin{proof}
Let $X$ be a smooth projective toric variety satisfying the condition stated in the proposition.  Consider the diagram
\begin{center}
\begin{tikzpicture}
\matrix[matrix of math nodes, row sep= 5ex, column sep= 4em, text height=1.5ex, text depth= .25ex]{
%%% row 1
						&
|(X)|   X		& 
						\\
%%% row 2
|(Y)| \overline Y		&		
			&
|(Z)| \overline Z		\\
			
%%% row 3	
					&
|(YcZ)| \overline{Y\cap   Z}  \ ,	&
					\\
};	
\begin{scope}[->,font=\footnotesize,auto]
\draw (Y)--node{$i$} (X);
\draw (Z)--node[swap]{$j$} (X);
\draw (YcZ)--node{$k$} (X);
\draw (YcZ)--node[swap]{$i'$} (Z);
\draw (YcZ)--node{$j'$} (Y);
\end{scope}
\end{tikzpicture}
\end{center}
where $i$, $j$, $i'$, $j'$, and $k$ are the natural closed immersions.
Since $\overline Y$ and $\overline Z$ meet properly in the smooth variety $\overline{Y\cap Z}$ we have $\mN_{\overline{Y\cap Z}/X}^\vee=(j')^*\mN_{\overline Y/X}\oplus (i')^*\mN_{\overline Z/X}$. Thus, if $c$ and $d$ denote the codimension of $Y$ and $Z$ in $T$, respectively, we have
\begin{multline*}
\bI_X^{-1}\chi_y^T(Y\cap Z)=k_*\bigg[\Big(y^{c+d}\lambda_{-y^{-1}}[\mN_{\overline{Y\cap Z}/X}^\vee]\Big)^{-1}\bigg]=\\
=k_*\bigg[(j')^*\Big(y^c\lambda_{-y^{-1}}[\mN_{\overline Y /X}^\vee]\Big)^{-1}\cdot (i')^*\Big(y^d\lambda_{-y^{-1}}[\mN_{\overline Z /X}^\vee]\Big)^{-1}\bigg] \ .
\end{multline*}
We see that it suffices to show that 
\begin{equation*}
k_*\Big[(j')^*\alpha)\cdot(i')^*\beta\Big]=(i_*\alpha)\cdot (j_*\beta)
\end{equation*}
 for all $\alpha\in K(\overline Y)$ and $\beta\in K(\overline Z)$. Using the projection formula, we see that the left hand side of this identity is equal to $i_*(\alpha\cdot j'_*(i')^*\beta)$, whereas the right hand side is equal to $i_*(\alpha\cdot i^*j_*\beta)$. Therefore, it is sufficient to prove that $i^*j_*=j'_*(i')^*$. But because $\overline Y$ and $\overline Z$ intersect transversally, they are Tor-independent over $X$, so we are done by \cite[IV 3.1.1]{SGA6}.
\end{proof}

\begin{cor}
Let $M$ be a lattice, let $\Delta_1,\ldots, \Delta_k$ be lattice polytopes in $M_\R$, and let $f_1,\ldots ,f_k\in \kappa[M]$ be generic with $\Newt(f_i)=\Delta_i$. Then for the complete intersection $Z=V(f_1)\cap\ldots\cap V(f_k)$ in $T=\Spec\kappa[M]$ we have
\begin{equation*}
\chi_y^T(Z)= \prod_{i=1}^k \frac{1-[\Delta_i]}{1-[\Delta_i]y}\ .
\end{equation*}
In particular, we have 
\begin{equation*}
\chi_y(Z) = \Lat\left((t-1)^n\prod_{i=1}^k \frac{1-[\Delta_i]}{1-[\Delta_i]y}\right)\ .
\end{equation*}
\end{cor}

\section{Refined Tropicalizations}

Recall that for every smooth variety $X$ there is a the Chern character $\ch\colon K(X)_\Q\to A^*(X)_\Q$, which is an isomorphism of rings \cite{F89}. It defines a natural transformation between $K$ and $A^*$ seen as contravariant functors from the category of smooth varieties to the category of rings. In particular, for a given lattice $M$ with dual $N$ it induces an isomorphism $\injlim_\Sigma K(X_\Sigma)_\Q \to \injlim_\Sigma A^*(X_\Sigma)_\Q$, where the colimits are taken over all smooth complete fans in $N_\R$, and $X_\Sigma$ denotes the toric variety associated to $\Sigma$. From what we saw in section \ref{subsec:KX for toric variety}, the colimit $\injlim_\Sigma K(X_\Sigma)_\Q$ over the Grothendieck rings is isomorphic to $\mL(M)_\Q$. On the other hand, the colimit $\injlim_\Sigma A^*(X_\Sigma)_\Q$ over the Chow rings is known to be isomorphic to the tropical intersection ring $Z^*(N_\R)$ \cite{FS97,AR10,Katz12,Rau08}. The isomorphism $\injlim_\Sigma A^*(X_\Sigma)\to Z^*(N_\R)$ can be described by writing $Z^*(N_\R)$ as a colimit $\injlim_\Sigma \Mink^*(\Sigma)$, where $\Mink^*(\Sigma)$ is the subgroup of $\Z^\Sigma$ of so-called Minkowski weights. We refer to \cite{FS97} for the precise definition of $\Mink^*(\Sigma)$. Thus, the tropical intersection ring $Z^*(N_\R)$ consists of weighted fans modulo refinements. The isomorphism $\injlim_\Sigma A^*(X_\Sigma)\to Z^*(N_\R)$ can be easily described: it assigns to a cocycle $C\in A^*(X_\Sigma)$ the class of the fan $\Sigma$ with weights given by $\Sigma\ni \sigma \mapsto \int_{X_\Sigma} C \cdot [V(\sigma)]$, where $V(\sigma)$ denotes the closed torus-invariant subvariety of $X_\Sigma$ corresponding to $\sigma$ under the Orbit-Cone Correspondence.    

Combining the three ring isomorphisms mentioned above, we obtain an isomorphism 
\begin{equation*}
\ch^{\trop}\colon \mL(M)_\Q\to Z^*(N_\R)_\Q   
\end{equation*}
which we will call the \emph{tropical Chern character}. It maps the generator of $\mL(M)_\Q$ corresponding to a lattice polytope $P$ to $\exp\mT(P)=\sum_{k=1}^\infty \frac1{k!}\mT(P)^k$, where $\mT(P)$ denotes the tropical cycle corresponding to $P$ (see \cite{JY16} for details). 

\begin{rem}
The isomorphism $\ch^{\trop}$ has already been considered by Jensen and Yu \cite{JY16} building on a result of Fulton and Sturmfels \cite{FS97} (which in turn relied on McMullen's work), albeit with a slightly different domain. Instead of $\mL(M)_\Q$ they consider McMullen's polytope algebra $\Pi(M_\Q)$, which is defined similarly as $\mL(M)$, but instead of being generated by lattice polytopes in $M$, it is generated by \emph{all} polytopes in the $\Q$-vector space $M_\Q$. This becomes a $\Q$-algebra $\widetilde\Pi(M_\Q)$ after replacing the $0$-graeded piece, which is isomorphic to $\Z$, by $\Q$ \cite{McMullen89}. They then show that there is an isomorphism $\widetilde\Pi(M_\Q)\to Z^*(N_\R)_\Q$ sending the class of any polytope $P$ in $M_\Q$ to $\exp \mT(P)$. This is one way to see that the canonical morphism $\mL(M)_\Q\to \widetilde\Pi(M_\Q)$ is an isomorphism, although this can be seen more directly \cite{McMullen09}. 
\end{rem}

Applying the tropical Chern character to the relative refined $\chi_y$-genus we obtain, while not really loosing any information, an object of a completely different flavor. We therefore give it different name:
\begin{defn}
Let $Z$ be a schön subvariety of an algebraic torus $T$. Then we define
\begin{equation*}
\Trop_y(Z)\coloneqq \ch^{\trop}(\chi_y^T(Z)) \in Z^*(N_\R)_\Q[y,(y-1)^{-1}]
\end{equation*}
and call it the refined tropicalization of $Z$.
\end{defn}

The polynomial ring $Z^*(N_\R)_\Q[y,(y-1)^{-1}$ is clearly isomorphic to $\injlim_\Sigma \Mink(\Sigma)\otimes_\Z \Q[y,(y-1)^{-1}]$. The latter is what one gets if one considers tropical cycles with weights in $\Q[y,(y-1)^{-1}]$ instead of $\Z$, so we denote this ring by $Z^*(N_\R; \Q[y,(y-1)^{-1}])$.

\begin{example}
\label{exmpl:formula}
With the formula of Remark \ref{rem:relative chiy of hypersurface as laurent polynomial in y-1}, we see that the refined tropicalization of a hypersurface $V(f)$, where $f$ is generic with Newton polytope $\Delta$, is
\begin{equation*}
\frac{1-\exp\mT(\Delta)}{1-\exp\mT(\Delta)y}=-\sum_{i=1}^n \left(\frac{\exp(-\mT(\Delta))-1}{y-1}\right)^i \ ,
\end{equation*}
where $n= \rk M$. It follows that in the case $M=\Z^2$ we only need to know $\mT(\Delta)$ and the self-intersection number $\mT(\Delta)^2=2\vol(\Delta)$ to compute $\Trop_y(V(f))$. We have depicted the refined tropicalizations for $\Delta_1=\conv\{0, (1,0),(0,1),(1,1)\}$ and $\Delta_2=\conv\{0,(2,0), (1,2) , (0,1)\}$ in Figure \ref{Figure:refined tropicalizations}.
\end{example}

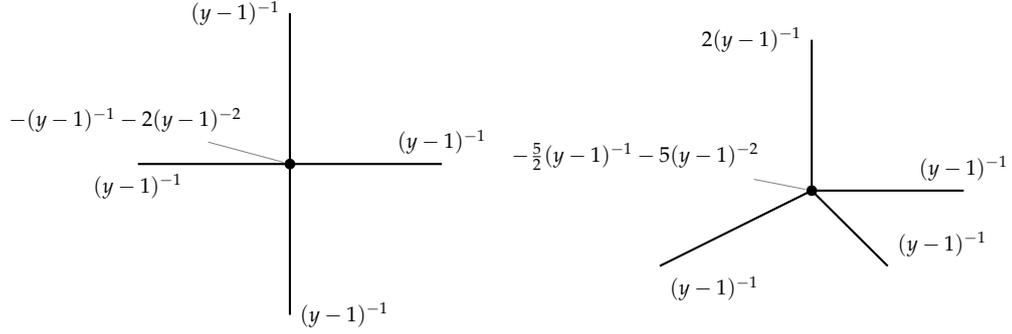
\begin{figure}
\newsavebox{\Newton}
\savebox{\Newton}{
\begin{tikzpicture}[thick,scale=.25]
\draw (0,0) -- (1,0) -- (1,1) -- (0,1) --cycle;
\end{tikzpicture}
}
\newsavebox{\Newtonn}
\savebox{\Newtonn}{
\begin{tikzpicture}[thick,scale=.15]
\draw (0,0) -- (2,0) -- (1,2) -- (0,1) -- cycle;
\end{tikzpicture}
}
\begin{minipage}{.5\linewidth}
\centering
\begin{tikzpicture}[thick,auto, font=\footnotesize]
\draw (0,0) --node[at end]{$(y-1)^{-1}$} (2,0);
\draw (0,0) --node[at end]{$(y-1)^{-1}$} (-2,0);
\draw (0,0) --node[at end]{$(y-1)^{-1}$} (0,2);
\draw (0,0) --node[at end]{$(y-1)^{-1}$} (0,-2);
\node [circle, fill=black, inner sep=.5mm, pin=150:{$-(y-1)^{-1}-2(y-1)^{-2}$}]at (0,0) {};
\end{tikzpicture}
\end{minipage}
\begin{minipage}{.5\linewidth}
\centering
\begin{tikzpicture}[thick,auto, font=\footnotesize]
\draw (0,0)-- node[at end]{$(y-1)^{-1}$}(2,0);
\draw (0,0) --node[at end]{$2(y-1)^{-1}$} (0,2);
\draw (0,0)-- node[at end]{$(y-1)^{-1}$}(-2,-1);
\draw (0,0) -- node[at end]{$(y-1)^{-1}$}(1,-1);
\node [circle, fill=black, inner sep=.5mm, pin=165:{$-\frac 52(y-1)^{-1}-5(y-1)^{-2}$}]at (0,0) {};
\end{tikzpicture}
\end{minipage}
\caption{The refined tropicalizations of generic hypersurfaces with Newton polytopes 
$\Delta_1=$\usebox{\Newton} and $\Delta_2=$\usebox{\Newtonn}, respectively.}
\label{Figure:refined tropicalizations}
\end{figure}

\begin{cor}

Let $Z$ and $Z'$ be two subvarieties of an algebraic torus $T$ that intersect generically. Then we have
\begin{equation*}
\Trop_y(Z\cap Z')=\Trop_y(Z)\cdot \Trop_y(Z')
\end{equation*}
\end{cor} 
\begin{proof}
This follows immediately from part (iii) of Theorem \ref{thm:main theorem} and the fact that $\ch^{\trop}$ is a morphism of rings.
\end{proof}

We remark that in the unrefined case the same equality holds for two not necessarily schön subvarieties $Z$ and $Z'$ of $T$ intersecting generically \cite{OP13}.

Of course, we need to show that $\Trop_y(Z)$ specializes to the ordinary tropicalization. Before stating this result, let us note two things: first, for every element $A\in Z^*(N_\R; \Q[y,(y-1)^{-1}])$ we can define its support $|A|$ analogously to the case where the weights are integers: it is the union of all cones with nonzero weight (cf.\ \cite{AR10}). Secondly, note that $Z^*(N_\R; \Q[y,(y-1)^{-1}])$ is graded by codimension. For $A\in Z^*(N_\R; \Q[y,(y-1)^{-1}])$, we denote by $0\neq A^{\top}$ the nonzero graded component of $A$ with minimal grade, that is the top-dimensional part of $A$.

\begin{prop}
\label{prop:Refined Trop specializes to unrefined Trop}
Let $Z$ be a codimension-$k$ schön subvariety of an algebraic torus $T$, and let $\Trop(Z)$ be its (ordinary) tropicalization. Then we have $|\Trop_y(Z)|=|\Trop(Z)|$ and 
\begin{equation*}
\Trop_0(Z)^{\top}=(-1)^k\Trop(Z) \ .
\end{equation*}
\end{prop}

\begin{proof}
It clearly suffices to show the second equality and the inclusion $|\Trop_y(Z)|\subseteq|\Trop(Z)|$. Let $X$ be a smooth projective toric variety such that $(X,Z)$ is almost tropical, and let $\overline Z$ denote the closure of $Z$ in $X$. Let $\Sigma$ denote the fan of $X$, and let $\sigma\in \Sigma$ be one of its cones. 
We have a commutative diagram
\begin{center}
\begin{tikzpicture}
\matrix[matrix of math nodes, row sep= 5ex, column sep= 4em, text height=1.5ex, text depth= .25ex]{
%%% row 1
|(VcZ)|  V(\sigma)\cap\overline Z		& 
|(VcZxT)|  \big(V(\sigma)\cap\overline Z\big)\times T 	&
|(V)| V(\sigma)		\\
%%% row 2	
 |(Z)| \overline Z 	&
|(ZxT)| \overline Z\times T		& 
|(X)|	X \ ,\\
};
\begin{scope}[->,font=\footnotesize,auto]
%%% horizontal
%% row 1
\draw (VcZ) --node{$\id\times 1$} (VcZxT);
\draw (VcZxT) -- (V);
\draw (VcZ) to[out=35, in=145]node{$i'$} (V);
%% row 2
\draw (Z) -- node{$\id\times 1$} (ZxT);
\draw (ZxT) -- (X);
\draw (Z) to[out=-35, in=-145]node[swap]{$i$} (X);
%%% vertical
\draw (VcZ)--node{$j'$} (Z);
\draw (VcZxT)--node{$j'\times \id$} (ZxT);
\draw (V)-- node{$j$}(X);
\end{scope}
\end{tikzpicture}
\end{center}
whose central squares are cartesian, where $i$, $j$, $i'$ and $j'$ denote the canonical closed immersions. As the multiplication map is smooth, $\overline Z\times T$ and $V(\sigma)$ are Tor-independent over $X$. It is also clear that $\overline Z$ and $ \big(V(\Sigma)\cap \overline Z\big) \times T$ are Tor-independent over $\overline Z\times T$. Thus, by \cite[III 1.5.1]{SGA6}, $\overline Z$ and $V(\sigma)$ are Tor-independent over $X$. 

By construction, the weight of  $\Trop_y(Z)$ at $\sigma$ equals
\begin{equation*}
\int_{V(\sigma)} j^*\ch\left(i_*\bigg[\Big(y^k\lambda_{-y^{-1}}[\mN_{\overline Z/X}^\vee]\Big)^{-1}\bigg]\right)=\int_{V(\sigma)}\ch \left( j^*i_*\bigg[\Big(y^k\lambda_{-y^{-1}}[\mN_{\overline Z/X}^\vee]\Big)^{-1}\bigg]\right) \ ,
\end{equation*}
where the equality holds because taking Chern characters commutes with taking pullbacks. Since $V(\sigma)$ and $\overline Z$ are Tor-independent over $X$, we have $j^* i_*=i_*(j')^*$ by \cite[IV 3.1.1]{SGA6}, so the above is equal to
\begin{equation*}
\int_{V(\sigma)}\ch \left( i'_*(j')^*\Big(y^k\lambda_{-y^{-1}}[\mN_{\overline Z/X}^\vee]\Big)^{-1}\right)=\int_{V(\sigma)}\ch \left( i'_*\bigg[\Big(y^k\lambda_{-y^{-1}}[(j')^*\mN_{\overline Z/X}^\vee]\Big)^{-1}\bigg]\right)
\end{equation*}
If $\sigma$ is not contained in $|\Trop(Z)|$, then $V(\sigma)\cap \overline Z$ is empty by \cite[Lemma 2.2]{Tev07}, and hence $[(j')^*\mN^\vee_{\overline Z/X}]$ and thereby the weight at $\sigma$ is zero. In particular, we know that every cone of dimension greater $\dim(Z)$ has weight $0$. Now suppose $\sigma\in \Sigma$ is a $\dim(Z)$-dimensional cone contained in $|\Trop(Z)|$. Plugging in $0$ in the formula above, we see that the weight of $\Trop_0(Z)$ at $\sigma$ is equal to 
\begin{equation*}
(-1)^k \int_{V(\sigma)} \ch \left(  i'_*\left[(j')^*\bigwedge\nolimits^k\mN_{\overline Z/X}\right]\right) \ .
\end{equation*} 
Because $Z$ is schön, the intersection $V(\sigma)\cap\overline Z$ is an integral $0$-dimensional scheme whose number of points is equal to the multiplicity of $\Trop(Z)$ at $\sigma$, which we denote by $m_Z(\sigma)$. Thus, $(j')^*\bigwedge\nolimits^k\mN_{\overline Z/X}$ is trivial, and the weight of $\Trop_0(Z)$ at $\sigma$ is equal to $(-1)^k m_Z(\sigma)$ times the degree of the Chern character of the structure sheaf of a point, which is $1$ by \cite[Example 15.2.16]{F89}. This finishes the proof.
\end{proof}

Given the refined tropicalization $\Trop_y(Z)$, it is nontrivial to recover $\chi_y^T(Z)$. It is therefore desirable to be able to read off $\chi_y(Z)$ directly from $\Trop_y(Z)$, without passing through $\chi_y^T(Z)$. But since we obtained $\Trop_y(Z)$ from $\chi_y^T(Z)$ by taking Chern characters coefficient-wise, we can use Hirzebruch-Riemann-Roch to do this. This depends, however, on the choice of a \emph{Todd measure}, as explained in the next paragraph.

Let $X$ be a complete toric variety corresponding to a fan $\Sigma$, and let $\mF$ be a vector bundle on $X$. Hirzebruch-Riemann-Roch tells us that 
\begin{equation*}
\chi(X,\mF)=\int_X \ch(\mF)\cap \Td(X) \ ,
\end{equation*}
where $\Td(X)$ is the Todd class of $X$. Both $\ch(\mF)$ and $\Td(X)$ can be described combinatorially: the Chern character $\ch(\mF)$ is completely determined by the Minkowski weight $\trop(\ch(\mF))$ on $\Sigma$ assigning to $\sigma\in\Sigma$ the  weight $\int_X \ch(\mF)\cap [V(\sigma)]$ ,whereas the Todd class can be written as $\Td(X)= \sum_{\sigma\in \Sigma} \mu^\Sigma(\sigma) [V(\sigma)]$ for appropriate $\mu^\Sigma(\sigma)\in \Q$. Hirzebruch-Riemann-Roch then yields the formula
\begin{equation*}
\chi(X,\mF)=\sum_{\sigma\in \Sigma} \trop(\ch(\mF))(\sigma) \cdot \mu^\Sigma(\sigma) \ .
\end{equation*}
Because of the resulting connection between counting lattice points and the Todd class of toric varieties Danilov posed the question whether the coefficients $\mu^\Sigma(\sigma)$ can be chosen independently of $\Sigma$ \cite{Danilov78}. This was answered positively, but non-constructively by Morelli \cite{Mor93b}, and later constructively by Pommersheim and Thomas \cite{PT04}. In fact, denoting the set of strongly convex rational polyhedral cones in $M_\R$ by $\mathcal C(M)$, both Morelli, and Pommersheim and Thomas showed that there exists a function $\mu\colon \mathcal C(M)\to \Q$ such that
\begin{enumerate}[label=(\roman*)]
\item  For every fan $\Sigma$ in $N_\R$ with associated toric variety $X_\Sigma$ we have 
\begin{equation*}
\Td(X_\Sigma)=\sum_{\sigma\in\Sigma}\mu(\sigma)[V(\sigma)] \ .
\end{equation*}
\item Whenever a cone $\tau\in\mathcal C(M)$ is the union of cones $\sigma_1\ldots \sigma_k$ of the same dimension, such that the $\sigma_i$ have disjoint relative interiors, we have
\begin{equation*}
\mu(\tau)=\sum_{i=1}^k\mu(\sigma_i) \ .
\end{equation*}
\end{enumerate}
Because of the latter additivity property any such $\mu$ is called a \emph{Todd measure}. Unfortunately, the Todd-measure is not uniquely determined by the two properties above, nor is there a canonical choice, unless one extends scalars (cf.\ \cite[Section 6]{PT04}). The additivity property makes it possible to integrate over tropical cycles:

\begin{defn}
Let $A\in Z^*(N_\R; G)$ be a tropical cycle with coefficients in an abelian group $G$, represented by a pair $(\Sigma,\omega)$, where $\omega\colon \Sigma\to G$ is a weight function on the fan $\Sigma$. Then we define
\begin{equation*}
\int A \,d\mu \coloneqq \sum_{\sigma\in\Sigma} \omega(\sigma) \mu(\sigma) \ .
\end{equation*}
Note that this is independent of the choice of fan structure by the additivity of $\mu$.
\end{defn}

Finally, part (i) of Theorem \ref{thm:main theorem}, together with what we mentioned above, yields the following corollary:

\begin{cor}
\label{cor: integral over refined trop yields chiy}
Let $Z$ be a schön subvariety of an algebraic torus $T$ with cocharacter lattice $N$. Let $\mu$ be a Todd measure on the strongly convex rational polyhedral cones in $N_\R$. Then we have
\begin{equation*}
\chi_y(Z)=\chi_y(T)\cdot\int\Trop_y(Z) \,d\mu \ .
\end{equation*}
\end{cor}

\begin{example}
Let us use the formula of Corollary \ref{cor: integral over refined trop yields chiy} to calculate the $\chi_y$-genera of two schön hypersurfaces $Z_1$ and $Z_2$ of $(\kappa^*)^2$  with  Newton polytopes $\Delta_1=\conv\{0, (1,0), (0,1),(1,1)\}$ and $\Delta_2=\conv\{0,(2,0) ,(1,2) ,(0,1)\}$, respectively. We have already calculated their refined tropicalizations using the formula of Example \ref{exmpl:formula} and have depicted them in Figure \ref{Figure:refined tropicalizations}. To choose a Todd measure $\mu$ we note that for every complete toric surface $X$ with fan $\Sigma$ we have 
\begin{equation*}
\Td(X)= 1+ \frac 12 \sum_{\rho} [D_\rho] + [X] \ ,
\end{equation*}
where the sum is over all rays $\rho$ of $\Sigma$, and $D_\rho$ is the torus-invariant prime divisor corresponding to $\rho$ by the Orbit-Cone correspondence. In light of this formula, we choose $\mu$ such that it assigns $1$ to the point, i.e\ the unique $0$-dimensional cone, and $\frac 12$ to every ray. Its values on the full-dimensional cones are less canonical, but depend on some choices as explained in detail in \cite{PT04}. Luckily, we will not actually need to know these values for our purposes. It is worth noting though that if one is willing to work with coefficients in $\R$, a suitable choice would be the normalized angle measure on the cocharacter lattice, which equals $\Z^2$ in our case. That is, one could choose $\mu$ such that it assigns to any full-dimensional cone $\frac 1{2\pi}$ times the angle around its apex with respect to the standard scalar product on $\R^2$.
Whatever our choice for the values of $\mu$ on the top-dimensional cones may be, knowing the values on the point and on the rays yields
\begin{equation*}
\chi_y(Z_1)=(y-1)^{2} \bigg[ \frac 12 \bigg(4\cdot (y-1)^{-1}\bigg) +1\cdot \bigg(-(y-1)^{-1}-2(y-1)^{-2} \bigg)\bigg]= 
y-3
\end{equation*}
and 
\begin{equation*}
\chi_y(Z_2)=(y-1)^{2}\bigg[
\frac 12 \bigg ( 5\cdot (y-1)^{-1}\bigg)+ 1\cdot \bigg (-\frac 52(y-1)^{-1}-5 (y-1)^{-2}\bigg)
\bigg]
= -5 \ .
\end{equation*}
More generally, if $Z$ is a schön hypersurface of $(\kappa^*)^2$, given by an equation with Newton polytope $\Delta\subseteq \Z^2$, then the formulas of Example \ref{exmpl:formula} and Corollary \ref{cor: integral over refined trop yields chiy} yield
\begin{equation*}
\chi_y(Z)=\frac { |\partial \Delta \cap \Z^2|-2\vol(\Delta)}2 (y-1) -2\vol(\Delta) \ ,
\end{equation*}
where $\partial\Delta$ denotes the boundary of $\Delta$, and the volume appears since the self-intersection $\mathcal T(\Delta)^2$ equals $2\vol(\Delta)$.
On the other hand, if $X$ is any smooth projective toric compactification of $(\kappa^*)^2$ such that $(X,Z)$ is almost tropical, then the genus of the smooth curve $\overline Z$ is $|\mathring \Delta\cap \Z^2|$ by \cite[Prop.\ 10.5.8]{CLS}, and the number of points of $\overline Z$ in the boundary of $X$ is equal to $|\partial \Delta\cap \Z^2|$ by \cite[Lemma 3.4.6]{TropBook} . Therefore, we also have
\begin{equation*}
\chi_y(Z) =\big(1-|\mathring\Delta \cap \Z^2|\big) (y+1) - |\partial \Delta\cap \Z^2| \ .
\end{equation*}
Equating the two formulas we recover Pick's theorem.
\end{example}

\bibliographystyle{../../TexStuff/amstest}
\bibliography{../../TexStuff/bib}

\normalsize
Andreas Gross, Imperial College London, Department of Mathematics, South Kensington
Campus, London SW7 2AZ, UK, \href{mailto:a.gross@imperial.ac.uk}{\ttfamily a.gross@imperial.ac.uk}

\end{document}